\documentclass[12pt]{amsart}
\usepackage{amsmath,mathtools,amsthm,amsfonts,bigints,amssymb, mathrsfs}

\newtheorem{theorem}{Theorem}[section]
\newtheorem{lemma}[theorem]{Lemma}

\newtheorem{corollary}[theorem]{Corollary}

\theoremstyle{definition}
\newtheorem{definition}[theorem]{Definition}

\theoremstyle{remark}
\newtheorem{remark}[theorem]{Remark}

\newcommand{\R}{\mathbb R}%
\newcommand{\M}{\mathcal M}%
\newcommand{\C}{\mathbb C}%
\newcommand{\N}{\mathbb N}%
\numberwithin{equation}{section}

\begin{document}
\title[Admissible and sectorial convergence on Harmonic $NA$ groups]{Admissible and sectorial convergence of generalized Poisson integrals on Harmonic $NA$ groups}
\author[Utsav Dewan]{Utsav Dewan}
\address{Stat-Math Unit, Indian Statistical
Institute, 203 B. T. Rd., Calcutta 700108, India}
\email{utsav97dewan@gmail.com}
\thanks{The author is supported by a research fellowship of Indian Statistical
Institute}

\begin{abstract} We prove a converse of Fatou type result for certain eigenfunctions of the Lalplace-Beltrami operator on Harmonic NA groups relating sectorial convergence and admissible convergence of Poisson type integrals of complex (signed) measures. This result extends several results of this kind proved eariler in the context of the classical upper half space $\R_+^{n+1}$. Similar results are also obtained in the degenerate case of the real hyperbolic spaces.
\end{abstract}
\maketitle

\section{Introduction}
In the celebrated article \cite{Fatou}, Fatou posed and solved the classical problem which relates the differentiability of a suitable Borel measure (possibly complex valued) $\mu$ on $\R$ with the boundary behaviour of the harmonic function $u=P[\mu]$ on the upper half-plane $\R^2_+$ obtained as the Poisson integral of $\mu$. More precisely, it was proved that the harmonic function $u$ will have non-tangential limit at some $x_0 \in \R$  if $\mu$ is differentiable at $x_0$. Loomis \cite{L} proved that the converse is true when $\mu$ is assumed to be positive. It was also shown by an example that the converse is false in the absence of  positivity of $\mu$.

Ramey and Ullrich \cite[Theorem 2.2]{UR} obtained higher dimensional generalizations of the result of Loomis by introducing the notion of strong derivative of a measure. Precisely, it was proved in \cite{UR} that the non-tangential convergence of the Poisson integral $P[\mu]:\R_+^{n+1}\to [0,\infty)$, to a boundary point $x_0\in\R^n$ is equivalent to the existence of the strong derivative of the measure $\mu$ at $x_0$. Among other things, their proof crucially uses the positivity of $\mu$. We refer the reader to \cite{Brossard} for a remarkable generalization of this result which deals with a more general class of measures. \cite[Theorem 2.2]{UR} has recently been extended to the class of Harmonic $NA$ groups in \cite{Ray}, for certain positive eigenfunctions of the Laplace-Beltrami operator $\mathcal{L}$.
\begin{theorem}[\cite{Ray}, Theorem 4.2] \label{result_on_admissiblelimit}
Suppose that $u$ is a positive eigenfunction of $\mathcal{L}$ in $S$ with eigenvalue $\beta^2 - \rho^2$, where $\beta >0$, with boundary measure $\mu$, and that $n_0 \in N, L \in [0,\infty)$. Then $\mu$ has strong derivative at $n_0$, equal to $L$ if and only if the function $(n,a) \mapsto a^{\beta - \rho} u(n,a)$ has admissible limit $L$ at $n_0$ .
\end{theorem}
Motivated by a result of Gehring \cite{Gehring}, Ramey and Ullrich also related the sectorial convergence and non-tangential convergence of Poisson integral of positive measure. We will need to introduce some notation to explain this result.
In $\R^3_+$ the non-tangential cone with vertex at $x_0 \in \R^2$ and aperture $\alpha >0$ is defined by
\begin{equation*}
\Gamma_{\alpha}(x_0)=\{(x,y) \in \R^3_+ : \|x-x_0\| < \alpha y\}.
\end{equation*}
The topological boundary $\partial \Gamma_{\alpha}(x_0)$ of $\Gamma_{\alpha}(x_0)$ with respect to $\R^3_+$ is denoted by
\begin{equation*}
\partial \Gamma_{\alpha}(x_0)=\{(x,y) \in \R^3_+ : \|x-x_0\| = \alpha y\}.
\end{equation*}
Then a sector $\mathscr S$ of $\partial \Gamma_{\alpha}(x_0)$ is defined to be the portion of $\partial \Gamma_{\alpha}(x_0)$ bounded by two rays in $\partial \Gamma_{\alpha}(x_0)$ emanating from $x_0\in\R^2$, that is
\begin{equation*}
\mathscr S =\bigcup\limits_{\theta\in (\theta_1,\theta_2)}\gamma_{x_0,\alpha,\theta}(0,\infty),\:\:\:\:\text{for some $\theta_1, \theta_2 \in [0,2\pi)$,}
\end{equation*}
where $\gamma_{x_0,\alpha,\theta}:(0,\infty)\to \R^3_+$ is defined by,
\begin{equation*}
\gamma_{x_0,\alpha,\theta}(y)=(x_0+ (\alpha y\cos\theta, \alpha y\sin\theta),y).
\end{equation*}
A function $u:\R^3_+\to \C$ is said to have a sectorial limit $L\in\C$ at $x_0$ if for some $\mathscr S\subset \partial\Gamma_{\alpha}(x_0)$
\begin{equation*}
\displaystyle\lim_{\substack{
z \to 0\\
z \in \mathscr{S}}}
u(z)=L.
\end{equation*}
In \cite[Theorem 3.3]{UR} it was proved that if a positive harmonic function $u$ on $\R^3_+$ has a sectorial limit $L \in [0,\infty)$ at $x_0$  then $u$ has non-tangential limit $L$ at $x_0$. Although the result was proved for $\R^3_+$, it can be easily generalized to $\R^{n+1}_+$ for $n \ge 3$. A variant of this result was proved by Saeki in \cite[Corollary 3.4]{Sa}. For a class of measures (not necessarily nonnegative), Saeki proved that if the Poisson integral $P[\mu]$ converges to $L$ along rays (emanating from a boundary point $x_0\in\R^n$) which form a dense subset of the topological boundary of some non-tangential cone in $\R^{n+1}_+$ then $P[\mu]$ has non-tangential limit $L$ at $x_0$.

A comparison between the result of Ramey and Ullrich regarding sectorial convergence of the Poisson integral and that of Saeki is as follows. Ramey and Ullrich dealt with nonnegative measures $\mu$ and assumed convergence of $P[\mu]$ on some sector of the topological boundary of a non-tangential cone which is not a dense subset of this topological boundary. However, the result of Saeki accommodates a larger class of measures $\mu$ but assumes convergence of $P[\mu]$ along rays on some dense subset of the topological boundary of a non-tangential cone.

In this paper we wish to prove a result which generalizes both the results stated above in the context of Damek-Ricci spaces or the so called Harmonic $NA$ groups. Instead of Poisson integral of nonnegative measures we deal with certain generalized Poisson integral (related to positive eigenfunctions of the Laplace-Beltrami operator) of measures on the boundary $N$ which satisfies a condition similar to that of Saeki. We will also introduce a notion analogous to that of a sector (see Definition \ref{sectordefn}) suitable for our context and prove that the boundary behaviour of the corresponding convolution integral along a sector is sufficient to conclude the existence of its admissible limit (as introduced in \cite{Kosym}). We refer the reader to Section $2$ for unexplained notations and terminologies. The following is our main result.
\begin{theorem}
\label{result_1}
For $\beta >0$, suppose $\mu \in \M_{\beta}$ is such that
\begin{equation}
\label{measure_condition}
\displaystyle\sup_{0<r<t_0} \frac{|\mu|(B(n_0,r))}{m(B(n_0,r))} < \infty \:,
\end{equation}
for some  $t_0>0$, for $n_0 \in N$. If $u=Q_{i \beta}[\mu]$ has a sectorial limit $L \in \mathbb{C} $ at $n_0$, then $u$ has admissible limit $L$ at $n_0$.
\end{theorem}
We refer the reader to Remark \ref{conditions_on_the_boundary_measure} regarding the condition imposed on the boundary measure in the theorem above which shows that this condition accommodates all the cases dealt with in \cite{UR, Brossard}  for Euclidean spaces. This paper is organised as follows. In section $2$, we will discuss the preliminaries about Harmonic NA groups, generalized Poisson kernel, Poisson integral on these groups and introduce the notion of sector and the sectorial limit. In section $3$ we prove some auxiliary results which play a crucial role in the proof of our main result. The proof of the main result is given in Section $4$. In Section $5$ we discuss some related results for the Real hyperbolic spaces.

\section*{Acknowledgements} The author  would  like to thank Prof. Swagato K. Ray for suggesting the problem. The author is grateful to him and to Prof. Kingshook Biswas for many illuminating discussions.

\section{Preliminaries of Harmonic NA groups}
In this section we will discuss the necessary preliminaries and fix our notations. Most of this material can be found in \cite{ ADY, ACD, FS, Ray}.

Let $\mathfrak n$ be an $H$-type Lie algebra. $\mathfrak z$ and $\mathfrak v$ will denote the center of $\mathfrak n$ and its orthogonal complement (with respect to the underlying inner product) respectively. The connected and simply connected Lie group $N$ corresponding to $\mathfrak n$ is called an $H$-type group. Since $\mathfrak{n}$ is nilpotent, the exponential map defines an analytic diffeomorphism and hence one gets the parametrization  $N=\exp\mathfrak{n}$ by $(X,Z)$ for $X\in\mathfrak{v},\:Z\in\mathfrak{z}$.  The group law of $N$ in the exponential coordinates is given by
\begin{equation*}
(X,Z)(X',Z')=(X+X',Z+Z'+\frac{1}{2}[X,X']).
\end{equation*}
The Haar measure of $N$ is given by $dXdZ$ and will be denoted by $m$. Let $S=NA$ be the semi-direct product of $N$ and the multiplicative group $A=(0,\infty)$ with respect to the nonisotropic dilation:
\begin{equation}\label{nonisotropic}
\delta_a(n)=\delta_a(X,Z)=(\sqrt{a}X,aZ),\:\:a\in A,\:n=(X,Z)\in N.
\end{equation}
Hence the multiplication on $S$ is given by
\begin{equation*}
(X,Z,a)(X',Z',a')=(X+\sqrt{a}X',Z+aZ'+\frac{1}{2}\sqrt{a}[X,X'],aa').
\end{equation*}
Then $S$ turns out to be a solvable Lie group having Lie algebra $\mathfrak{s}=\mathfrak v \oplus\mathfrak{z}\oplus\R$ with Lie bracket
\begin{equation*}
[(X,Z,u),(X',Z',u')]=(\frac{1}{2}uX'-\frac{1}{2}u'X,uZ'-u'Z+[X,X'],0).
\end{equation*}

We shall write $(n,a)=(X,Z,a)$ for the element $\left(\exp(X+Z),a\right),\:a\in A,\:X\in\mathfrak{v},\:Z\in\mathfrak{z}$ and thus identify $S$ with the upper-half space $\R^{2p+k+1}_+$ topologically, where $dim\:\mathfrak{v}=2p,\:dim\:\mathfrak{z}=k$. Setting $Q=p+k$, one can see that
\begin{equation}\label{measuredilation}
m\left(\delta_a(E)\right)=a^Qm(E),
\end{equation}
is true for all measurable sets $E\subseteq N$ and $a\in A$. This $Q$ is called the homogeneous dimension of $N$. The group $S$ is equipped with the left-invariant Riemannian metric induced by
$$\langle (X,Z,u),(X',Z',u')\rangle_S=\langle X,X'\rangle+\langle Z,Z'\rangle+uu'$$
on $\mathfrak{s}$. Then the associated left-invariant Haar measure $dx$ on $S$ is given by
\begin{equation*}
dx=a^{-Q-1}dXdZda,
\end{equation*}
where $dX,\:dZ,\:da$ are the Lebesgue measures on $\mathfrak{v},\:\mathfrak{z},\:A$ respectively. Let $\rho=Q/2$ and we note that $\rho$ corresponds to the half-sum of positive roots when $S=G/K$, is a rank one symmetric space of noncompact type. It is noteworthy however that rank one symmetric spaces of noncompact type form a very small subclass of the class of Harmonic $NA$ groups.
We will use $e$ to denote the identity element $(\underline{0},1)$ of $S$, where $\underline{0}$ and $1$ are the identity elements of $N$ and $A$ respectively.
$N$ being a stratified Lie group, admits homogeneous norms with respect to the family of dilations $\{\delta_a\mid a\in A\}$ \cite[P.8-10]{FS}. We recall that a continuous function $d:N\to[0,\infty)$ is said to be a homogeneous norm on $N$ with respect to the family of dilations $\{\delta_a\mid a\in A\}$ if $d$ satisfies the following \cite[P.8]{FS}:
\begin{enumerate}
	\item[(i)]$d$ is smooth on $N\setminus\{0\}$;
	\item[(ii)] $d(\delta_a(n))=ad(n)$,  for all  $a>0,\:n\in N$;
	\item [(iii)]$d(n^{-1})=d(n)$, for all  $n\in N$;
	\item[(iv)]$d(n)=0$  if and only if  $n=\underline{0}$.
\end{enumerate}
From \cite[Proposition 1.6]{FS}, we know that for a homogeneous norm $d$ on $N$, there exists a positive constant $\tau_d \geq1$, such that
\begin{equation}\label{quasinorm}
d(nn_1)\leq \tau_d\left[d(n)+d(n_1)\right],\;\:\:\:n\in N, n_1\in N.
\end{equation}
A homogeneous norm $d$ defines a left invariant quasi-metric on $N$, denoted by $\bf{d}$, as follows:
\begin{equation*}
{\bf d}(n_1,n_2)=d(n_1^{-1}n_2),\:\:\:\: n_1\in N,\:n_2\in N.
\end{equation*}
From the definition of $d$ and from (\ref{quasinorm}), the following properties of the associated quasi-metric easily follow:
\begin{enumerate}
\item [i)] ${\bf d}(n_1,n_2)={\bf d}(n_2,n_1),\:\:\:\text{for all}\:\:n_1,\:n_2\in N$.
\item[ii)] ${\bf d}(nn_1,nn_2)={\bf d}(n_1,n_2),\:\:\:\text{for all}\:\:n_1,\:n_2,\:n\in N.$	
\item [iii)] For all $n_1,\:n_2,\:n\in N$
\begin{equation}\label{quasitriangle}
{\bf d}(n_1,n_2)\leq \tau_d\left[{\bf d}(n_1,n)+{\bf d}(n,n_2)\right].
\end{equation}
\end{enumerate}
Any two homogeneous norms $d_1$ and $d_2$ on $N$ are equivalent \cite[P.230]{BLU}. The homogeneous we will be working with is the following \cite[P.1918]{DK}:
\begin{equation}\label{norm}
d(n)=d(X,Z)=(\|X\|^4+16\|Z\|^2)^{\frac{1}{2}},\:\:\:\:\:n=(X,Z)\in N,
\end{equation}
where $\|X\|$, $\|Z\|$ are the usual Euclidean norms of $X\in\mathfrak{v}\cong\R^{2p}$ and $Z\in\mathfrak{z}\cong\R^k$ respectively. For $n\in N$ and $r>0$,  the $d$-ball centered at $n$ with radius $r$ is defined as
\begin{equation*}
B(n,r)=\{n_1\in N\mid{\bf d}(n,n_1)<r\}=\{n_1\in N\mid d(n_1^{-1}n)<r\}.
\end{equation*}
Then $\overline{B}(n,r)=\{n_1\in N\mid {\bf d}(n,n_1)\leq r \}$, is a compact subset of $N$ \cite[Lemma 1.4]{FS}.
So $B(n,r)$ is the left translate by $n$ of the ball $B(\underline{0},r)$, which is the image under the dilation $\delta_r$ of the unit ball $B(\underline{0},1)$. It is also easy to observe that if $B=B(n,t)$ for some $n\in N$, $t>0$, then
\begin{equation*}
\delta_r(B)=B(\delta_r(n),rt),\:\:\:\:\:\text{for all $r>0$}.
\end{equation*}
 For a function $\psi$ defined on $N$, we define for $a>0$,
\begin{equation}\label{dilationoffunction}
\psi_a(n)=a^{-Q}\psi\left(\delta_{{\frac{1}{a}}}(n)\right),\:\:\:\:\:n\in N.
\end{equation}
If $g$ is a measurable function on $N$ and $\mu$ is a measure on $N$, their convolution $\mu\ast g$ is defined by
\begin{equation*}
\mu\ast g(n)=\int_{N}g(n_1^{-1}n)\:d\mu(n_1),
\end{equation*}
provided the integral converges.

We now describe the Laplace-Beltrami operator $\mathcal{L}$ on $S$. Let $\{e_i\mid 1\leq i\leq 2p\}$, $\{e_r\mid 2p+1\leq r\leq 2p+k\}$, $\{e_0\}$ be an orthonormal basis of $\mathfrak{s}$ corresponding to the decomposition $\mathfrak{s}=\mathfrak{n}\oplus\mathfrak{z}\oplus\R$. $E_{l}$ will denote the left-invariant vector field on $S$ determined by $e_{l}$, $0\leq l\leq 2p+k$. From (\cite[Theorem 2.1]{D1},\cite[P. 234]{DR}) it is known that the Laplace-Beltrami operator $\mathcal{L}$ associated to the left-invariant metric $\langle,\rangle_S$ is the form
\begin{equation*}
\mathcal{L}=\sum_{l=0}^{2p+k}E_{l}^2-QE_0.
\end{equation*}
Let $\partial_i,\:\partial_r,\:\partial_a$ be the partial derivatives for the system of coordinates $(X_i,Z_r,a)$ corresponding to $(e_i,e_r,e_0)$. From the definition of left-invariant vector fields it can be shown that
\begin{eqnarray*}
&E_0& =a\partial_a,\\
&E_i&=a\partial_i+\frac{a}{2}\sum_{r=2p+1}^{2p+k}\langle[X,e_i],e_r\rangle\partial_r,\\
&E_r&=a^2\partial_r,\:\:\text{ for }1\leq i\leq 2p,\:\:2p+1\leq r\leq 2p+k.
\end{eqnarray*}
Then $\mathcal{L}$ can be written as \cite[P.234]{DR}
\begin{equation}\label{laplacebeltramionna}
\mathcal{L}=a^2\partial_a^2+\mathcal{L}_a+(1-Q)a\partial_a,
\end{equation}
where \begin{equation*} 
\mathcal{L}_a=a(a+\frac{1}{4}\|X\|^2)\sum_{r=2p+1}^{2p+k}\partial_r^2+a \sum_{i=1}^{2p}\partial_i^2+a^2\sum_{2p+1}^{2p+k}\sum_{i=1}^{2p}\langle[X,e_i],e_r\rangle\partial_r\partial_i \:.
\end{equation*}
The Poisson kernel $\mathcal{P}:S\times N\to\C$, corresponding to $\mathcal{L}$ is given by \cite[P.409]{ACD}
\begin{equation*}
\mathcal{P}(x,n)=P_a(n_1^{-1}n),\:\:\:\:\:x=(n_1,a)\in S,\:\:n\in N,
\end{equation*}
where $P$ is the function on $N$ defined by
\begin{equation*} 
P(n)=P(X,Z)=\frac{c_{p,k}}{\left(\left(1+\frac{\|X\|^2}{4}\right)^2+\|Z\|^2\right)^Q},\:\:\:\:n=(X,Z)\in N,
\end{equation*}
and $c_{p,k}$ is a positive constant so that $\|P\|_{L^1(N)}=1$. Then dilation of $P$ (see (\ref{dilationoffunction})) is given by,
\begin{equation}\label{poissonna}
P_a(n)=P_a(X,Z)=\frac{c_{p,k}\:a^Q}{\left(\left(a+\frac{\|X\|^2}{4}\right)^2+\|Z\|^2\right)^Q},\:\:\:\:n=(X,Z)\in N,\:\:a>0.
\end{equation}
Simplifying the expression (\ref{poissonna}), we get that
\begin{equation}\label{alternativepoisson}
P_a(n)=P_a(X,Z)=\frac{16^Qc_{p,k}\:a^Q}{\left(16a^2+8a\|X\|^2+d(X,Z)^2\right)^Q},\:\:\:\:n=(X,Z)\in N,\:\:a>0.
\end{equation}
The more general eigenfunctions of $\mathcal{L}$ can be obtained in the following manner: for $\lambda\in\C$, the $\lambda$-Poisson kernel is defined as
\begin{equation}\label{lambdapoisson}
\mathcal{P}_{\lambda}(x,n)=\left[\frac{\mathcal{P}(x,n)}{P(\underline{0})}\right]^{\frac{1}{2}-\frac{i\lambda}{Q}}=\left[\frac{P_a(n_1^{-1}n)}{P(\underline{0})}\right]^{\frac{1}{2}-\frac{i\lambda}{Q}},\:\:\:\:x=(n_1,a)\in S,\:\:n\in N.
\end{equation}
For $\lambda\in\C$, the function $\mathcal{P}_{\lambda}(.,n)$ is an eigenfunction of $\mathcal{L}$ with eigenvalue $-(\lambda^2+\rho^2)$, for each fixed $n\in N$, \cite[P.654]{ADY}.
One can normalize $\mathcal P_{\lambda}$, for $\text{Im}(\lambda)>0$, to define
\begin{equation}\label{normalisedpoisson}
\tilde{\mathcal{P}_{\lambda}}(x,n)=C_{\lambda}\mathcal{P}_{\lambda}(x,n),\:\:\:x\in S,\:n\in N,
\end{equation}
where $C_{\lambda}=c_{p,k}{\bf c}(-\lambda)^{-1}$ and {\bf c} is the generalization of the Harish-Chandra {\bf c}-function. For $\text{Im}(\lambda)>0$, the $\lambda$-Poisson transform of a measure $\mu$ on $N$ is defined by
\begin{equation}\label{plambdamu}
\mathcal{P}_{\lambda}[\mu](n,a)=\int_{N}\tilde{\mathcal{P}_{\lambda}}((n,a),n')\:d\mu(n'),
\end{equation}
whenever the integral converges absolutely for every $(n,a)\in S$. In this case, we say that the $\lambda$-Poisson transform $\mathcal{P}_{\lambda}[\mu]$ of $\mu$ is well-defined. Since for each $n\in N$, $\tilde{\mathcal{P}_{\lambda}}(.,n)$ is an eigenfunction of $\mathcal{L}$ with eigenvalue $-(\lambda^2+\rho^2)$, it follows that $\mathcal{P}_{\lambda}[\mu]$ is also an eigenfunction of $\mathcal{L}$ with the same eigenvalue provided $\mathcal{P}_{\lambda}[\mu]$ is well-defined. Using the definition of $\mathcal{P}_{\lambda}$ given in (\ref{lambdapoisson}) and the relation (\ref{normalisedpoisson}), it can be easily shown that for $\text{Im}(\lambda)>0$, $x=(n_1,a)\in S$ and $n\in N$,
\begin{equation}\label{relationbetweenpandq}
\tilde{\mathcal{P}_{\lambda}}(x,n)=a^{\frac{Q}{2}+i\lambda}q^{\lambda}_a(n_1^{-1}n)
\end{equation}
where
\begin{equation*} 
q^{\lambda}(n):=C_{\lambda}\left[P(\underline{0})^{-1}P(n)\right]^{\frac{1}{2}-\frac{i\lambda}{Q}},\:\:\: q^{\lambda}_a(n)=a^{-Q}q^{\lambda}\left(\delta_{a^{-1}}(n)\right),
\end{equation*}
$\text{Im}(\lambda)>0$, $n\in N$ and $a\in A$. An explicit expression of the function $q_a^{\lambda}$, $\text{Im}(\lambda)>0$, obtained from (\ref{alternativepoisson}) is given by,
\begin{equation*} 
q^{\lambda}_a(n)=c_{\lambda}\:\frac{a^{-2i\lambda}}{\left(16a^2+8a\|X\|^2+d(n)^2\right)^{\frac{Q}{2}-i\lambda}},\:\:\:n=(X,Z)\in N,\:a\in A,
\end{equation*}
where $c_{\lambda}=16^{\rho-i\lambda}c_{p,k}{\bf c}(-\lambda )^{-1}$.
If $\text{Im}(\lambda)>0$, then it is easily follows that
\begin{equation*}
\int_{N}q^{\lambda}(n)\:dm(n)=1.
\end{equation*}
 For a measure $\mu$ on $N$ and $\text{Im}(\lambda)>0$, we define the convolution integral
\begin{equation*} 
\mathcal{Q}_{\lambda}[\mu](n,a):=\mu\ast q^{\lambda}_a(n)=a^{-Q}\int_{N}q^{\lambda}\left(\delta_{a^{-1}}(n_1^{-1}n)\right)\:d\mu(n_1),
\end{equation*}
whenever the integral converges absolutely for every $(n,a)\in S$ and say that $\mathcal{Q}_{\lambda}[\mu]$ is well-defined. From the definition of the $\lambda$-Poisson integral (\ref{plambdamu}) and (\ref{relationbetweenpandq}), it follows that for a measure $\mu$ with well-defined $\mathcal{P}_{\lambda}[\mu]$ with $\text{Im}(\lambda)>0$,
\begin{equation}\label{plamdaandqlamda}
\mathcal{P}_{\lambda}[\mu](n,a)=a^{\frac{Q}{2}+i\lambda}\mathcal{Q}_{\lambda}[\mu](n,a),\:\:\:\text{for all}\:\:(n,a)\in S.
\end{equation}
In this paper we will be interested only in the case $\lambda=i\beta$, for $\beta>0$.
\begin{definition}
 For $\beta >0$, $\M_{\beta}$ is defined to be the set of all radon measures $\mu$ on $N$ such that $\mathcal{P}_{i\beta}[|\mu|]$ (equivalently $\mathcal{Q}_{i\beta}[|\mu|]$) is well-defined.
\end{definition}
  If $\mu\in \M_{\beta}$ then for all $n=(X,Z)\in N$, $a\in A$, we have the explicit expression
\begin{equation} \label{explicitqbetamu}
\mathcal{Q}_{i\beta}[\mu](n,a) = c_{\beta}\:a^{2\beta}\int_{N}\frac{1}{\left(16a^2+8a\|X-X_1\|^2+d\left((X_1,Z_1)^{-1}(X,Z)\right)^2\right)^{\rho+\beta}}\:d\mu(X_1Z_1),
\end{equation}
where
\begin{equation*}
c_{\beta}=16^{\rho+\beta}C_{i\beta}=16^{\rho+\beta}c_{p,k}{\bf c}(-i\beta)^{-1}>0.
\end{equation*}
We now state some important definitions.
\begin{definition}\label{basicdefn}
\begin{enumerate}
\item[i)] A function $u$ defined on $S$ is said to have admissible limit $L\in\C$, at $n_0\in N$, if for each $\alpha>0$
\begin{equation*}
\lim_{\substack{a\to 0\\(n,a)\in \Gamma_{\alpha}(n_0)}}u(n,a)=L,
\end{equation*}
where
\begin{equation}\label{admissibledomain}
\Gamma_{\alpha}(n_0):=\{(n,a)\in S\mid{\bf d}(n_0,n)<\alpha a\}=\{(n,a)\in S\mid d(n_0^{-1}n)<\alpha a\}
\end{equation}
is called the admissible domain with vertex at $n_0$ and aperture $\alpha$.
\item[ii)] Given a measure $\mu$ on $N$, we say that $\mu$ has strong derivative $L\in[0,\infty)$ at $n_0\in N$ if
\begin{equation*}
\lim_{r\to 0}\frac{\mu(n_0\delta_r(B))}{m(n_0\delta_r(B))}=L,
\end{equation*}
holds for every $d$-ball $B\subset N$. The strong derivative of $\mu$ at $n_0$, if it exists, is denoted by $D\mu(n_0)$.
\item[iii)] A sequence of complex valued functions $\{u_j\}_{j\in\N}$ defined on $S$ is said to converge normally to a function $u$ if $\{u_j\}$ converges to $u$ uniformly on all compact subsets of $S$.
\item[iv)] A sequence of complex valued functions $\{u_j\}_{j\in\N}$ defined on $S$ is said to be locally bounded if given any compact set $K\subset S$, there exists a positive constant $C_K$ (depending only on $K$) such that for all $j\in\N$ and all $x\in K$,
\begin{equation*}
|u_j(x)|\leq C_K.
\end{equation*}
\item [v)] For a differential operator $D$ on $S$, a smooth function $u$ on $S$ satisfying $Du=0$ is said to be a $D$-harmonic function.
\end{enumerate}
\end{definition}
For the notion of admissible convergence in the context of Riemannian symmetric spaces of noncompact type with real rank one we refer the reader to \cite[Page  158]{KP}.
Our main aim would be to conclude existence of admissible limit from the weaker assumption of existence of limits along some curves lying on
\begin{equation*}
\partial \Gamma_{\alpha}(n_0):=\{(n,a)\in S\mid{\bf d}(n_0,n)=\alpha a\}=\{(n,a)\in S\mid d(n_0^{-1}n)=\alpha a\} \:.
\end{equation*}
Then using the explicit expression (\ref{norm}), for $n_0=\underline{0}$, we can rewrite $\partial \Gamma_{\alpha}(\underline{0})$ as,
\begin{equation*}
\partial \Gamma_{\alpha}(\underline{0})= \{(\sqrt{\alpha a|\cos \theta|}\omega, \frac{1}{4} \alpha a |\sin \theta| \zeta,a): \omega \in S^{2p-1}, \zeta \in S^{k-1}, \theta \in [0,2 \pi),a>0\}
\end{equation*}
where $S^{m-1}$ denotes the unit sphere of $\R^m$, with respect to the Euclidean norm. Now we define the curves we are interested in.
\begin{definition}
\label{sectordefn}
\begin{enumerate}
\item[i)] For a fixed 4-tuple $(\alpha, \omega, \zeta, \theta)$ where $\alpha \in (0, +\infty), \: \omega \in S^{2p-1}, \: \zeta \in S^{k-1},\: \theta \in [0,2 \pi)$, the ray emanating from $\underline{0}$ is defined by the curve,
\begin{eqnarray*}
&& \gamma_{\underline{0},\alpha,\omega,\zeta,\theta} : (0,\infty) \to S, \:\:  \\
&& \gamma_{\underline{0},\alpha,\omega,\zeta,\theta}(a):= (\sqrt{\alpha a|\cos \theta|}\omega, \frac{1}{4} \alpha a |\sin \theta| \zeta,a) \:.
\end{eqnarray*}
Note that for $\alpha>0$,
\begin{equation}
\label{union}
\partial \Gamma_{\alpha}(\underline{0})= \bigcup\limits_{\omega \in S^{2p-1}} \: \bigcup\limits_{\zeta \in S^{k-1}} \:\bigcup\limits_{0 \le \theta < 2 \pi} \:\gamma_{\underline{0},\alpha,\omega,\zeta,\theta}(0,\infty) \:.
\end{equation}
For a general $n_0 \in N$, we translate our rays by $n_0$. More precisely, for a fixed 4-tuple $(\alpha, \omega, \zeta, \theta)$ where $\alpha \in (0, +\infty), \: \omega \in S^{2p-1}, \: \zeta \in S^{k-1},\: \theta \in [0,2 \pi)$, the ray emanating from $n_0$ is defined by,
\begin{eqnarray*}
&& \gamma_{n_0,\alpha,\omega,\zeta,\theta} : (0,+\infty) \to S, \:\:   \\
&& \gamma_{n_0,\alpha,\omega,\zeta,\theta}(a):= (n_0(\sqrt{\alpha a|\cos \theta|}\omega, \frac{1}{4} \alpha a |\sin \theta| \zeta),a) \:.
\end{eqnarray*}
By the left invariance of the quasi-metric it follows that, each $\gamma_{n_0,\alpha,\omega,\zeta,\theta} \in \partial \Gamma_{\alpha}(n_0)$ and in fact,
\begin{equation*}
\partial \Gamma_{\alpha}(n_0)= \bigcup\limits_{\omega \in S^{2p-1}} \: \bigcup\limits_{\zeta \in S^{k-1}} \:\bigcup\limits_{0 \le \theta < 2 \pi} \:\gamma_{n_0,\alpha,\omega,\zeta,\theta}(0,\infty) \:.
\end{equation*}
\item[ii)] A function $u$ defined on $S$ is said to have a radial limit $L \in \C$ at $n_0 \in N$ if
\begin{equation*}
\displaystyle\lim_{a \to 0} u(\gamma_{n_0,\alpha,\omega,\zeta,\theta}(a))= L \:,
\end{equation*}
for some ray $\gamma_{n_0,\alpha,\omega,\zeta,\theta}$.
\item[iii)] A sector of $\partial \Gamma_{\alpha}(n_0)$ is defined to be
\begin{equation*}
\mathscr S =  \bigcup\limits_{\omega \in \mathcal O_1} \: \bigcup\limits_{\zeta \in \mathcal O_2} \:\bigcup\limits_{\theta \in \mathcal O_3} \:\gamma_{n_0,\alpha,\omega,\zeta,\theta}(0,\infty) \:,
\end{equation*}
where $\mathcal O_1,\mathcal O_2,\mathcal O_3$ are relatively open subsets of $S^{2p-1},S^{k-1},[0,2 \pi)$ respectively. Note that such a sector becomes a relatively open subset of $\partial \Gamma_{\alpha}(n_0)$.
\item[iv)]  A function $u$ defined on $S$ is said to have a sectorial limit $L \in \C$ at $n_0 \in N$, if $u$ has radial limit $L$ along all rays in a sector of $\partial \Gamma_{\alpha}(n_0)$, for some $\alpha >0$ \:.
\end{enumerate}
\end{definition}

\section{some auxiliary results}
\begin{definition}
Given $\beta>0$, we define a second order differential operator $\mathcal{L}^{\beta}$ on $S$ having the same second order term as the Laplace-Beltrami operator $\mathcal{L}$ by the formula
\begin{equation*}
\mathcal{L}^{\beta}=a^2\partial_a^2+\mathcal{L}_a+(1-2\beta)a\partial_a,
\end{equation*}
where $\mathcal{L}$ is given by (\ref{laplacebeltramionna}).
\end{definition}
When $\beta=\rho=Q/2$, we recover $\mathcal{L}$ (see \cite[Theorem 3.2]{DK}). We note that
\begin{equation}\label{diflandlbeta}
\mathcal{L}^{\beta}-\mathcal{L}=2(\rho-\beta)a\partial_a=2(\rho-\beta)E_0.
\end{equation}
We recall that $E_0=a\partial_a$ is the left-invariant vector field on $S$ corresponding to the basis element $e_0=(0,0,1)$ of $\mathfrak{s}$ and hence $\mathcal{L}^{\beta}$ is left $S$-invariant. The following lemma shows that there is a one to one correspondence between the eigenfunctions of $\mathcal{L}$ with eigenvalue $\beta^2-\rho^2$ and $\mathcal{L}^{\beta}$-harmonic functions (as defined in Definition \ref{basicdefn}, v)).
\begin{lemma}[\cite{Ray}, Lemma 3.4]
\label{lbetaharmonic}
Let $\beta>0$ and let $u$ be a smooth function on $S$. Then $u$ is an eigenfunction of $\mathcal{L}$ with eigenvalue $\beta^2-\rho^2$, if and only if the function $(n,a)\mapsto a^{\beta-\rho}u(n,a)$, is $\mathcal{L}^{\beta}$-harmonic.
\end{lemma}
From (\ref{diflandlbeta}) we see that $\mathcal{L}^{\beta}$ and $\mathcal{L}$ differ by a first order term and hence $\mathcal{L}^{\beta}$ is an ellliptic operator. Application of a vast generalization of Montel's theorem, valid for solutions of hypoelliptic operators (and hence for $\mathcal{L}^{\beta}$-harmonic functions) proved in \cite[Theorem 4]{B}, yields the following result.
\begin{lemma}
\label{our_bar}
Let $\beta>0$, and let $\{F_j\}$ be a sequence of $\mathcal{L}^{\beta}$-harmonic functions on $S$. If $\{F_j\}$ is locally bounded then it has a subsequence which converges normally to a $\mathcal{L}^{\beta}$-harmonic function $F$.
\end{lemma}
We recall that if $\beta>0$, and $\mu\in M_{\beta}$ is positive then $\mathcal{P}_{i\beta}[\mu]$ is a positive eigenfunction of the Laplace-Beltrami operator $\mathcal{L}$ with eigenvalue $\beta^2-\rho^2$. Characterization of such positive eigenfunctions was proved by Damek and Ricci in \cite[Theorem 7.11]{DR}.
\begin{lemma}
\label{positiveeigen}
Suppose $u$ is a positive eigenfunction of the Laplace-Beltrami operator $\mathcal{L}$ on the Harmonic $NA$ group $S$ with eigenvalue $\beta^2-\rho^2$ for some $\beta>0$. Then there exists a unique positive measure $\mu$ on $N$ and a unique nonnegative constant $C$ such that
\begin{equation}\label{repofu}
u(n,a)=Ca^{\beta+\rho}+\mathcal{P}_{i\beta}[\mu](n,a),\:\:\:\:\text{for all}\:\:(n,a)\in S.
\end{equation}
In this case, the measure $\mu$ is called the boundary measure of $u$.
\end{lemma}
\begin{remark}\label{invarianceofadmissible}
Next, we consider the natural action of the subgroup $A$ on $S$ (see (\ref{nonisotropic})):
\begin{equation}\label{aaction}
r\cdot (n,a)=(\delta_r(n),ra),\:\:\:\:r\in A,\:(n,a)\in S.
\end{equation}
We note that each ray (see definition \ref{sectordefn}) emanating from $\underline{0}$ is invariant under this action. Indeed for every $r>0$,
\begin{eqnarray*}
&& r\cdot \gamma_{\underline{0},\alpha,\omega,\zeta,\theta}(a) \\
& = & (\delta_r(\sqrt{\alpha a|\cos \theta|}\omega, \frac{1}{4} \alpha a |\sin \theta| \zeta),ra) \\
& = & (\sqrt{r\alpha a|\cos \theta|}\omega, \frac{1}{4} r\alpha a |\sin \theta| \zeta ,ra) \\
&=& \gamma_{\underline{0},\alpha,\omega,\zeta,\theta}(ra) \:,
\end{eqnarray*}
for all $\alpha >0$, for all $\omega \in S^{2p-1}$, for all $\zeta \in S^{k-1}$ and for all $\theta \in [0,2 \pi )$. Then by (\ref{union}), each $\partial \Gamma_{\alpha}(\underline{0})$ being a union of such rays, is invariant under this action. Then since for any $\alpha >0$,
\begin{equation*}
\Gamma_{\alpha}(\underline{0}) = \{(\underline{0},a): a \in (0, +\infty)\} \bigcup \bigcup\limits_{0< \alpha' <\alpha} \partial \Gamma_{\alpha'}(\underline{0}) \:,
\end{equation*}
we get that each admissible domain $\Gamma_{\alpha}(\underline{0})$,with vertex at $\underline{0}$, is also invariant under this action.
\end{remark}

Given a function $F$ on $S$ and $r>0$, we define the dilation $F_r$ of $F$ by
\begin{equation*}
F_r(n,a):=F\left(\delta_r(n),ra\right),\:\:\:\:(n,a)\in S.
\end{equation*}
The class of $\mathcal{L}^{\beta}$-harmonic functions are preserved under this dilation.
\begin{lemma}[\cite{Ray}, Lemma 3.9]
\label{dilationof_Lbeta}
Let $\beta>0$. If $F$ is an $\mathcal{L}^{\beta}$-harmonic functions on $S$ then so is $F_r$, for every $r>0$.
\end{lemma}
Following \cite{Ray} we define the Hardy-Littlewood maximal function $M_{HL}(\mu)$ of $\mu$ by,
\begin{equation*}
M_{HL}(\mu)(n)= \displaystyle\sup_{r>0} \frac{|\mu|(B(n,r))}{m(B(n,r))} , \: n \in N.
\end{equation*}
\begin{lemma}
\label{HL}
For $\beta >0$, if $\mu \in M_{\beta}$ then for all $\alpha >0$, for all $\eta >0$, there exist positive constants $ C_{\alpha,\beta}$ and $C'_{\eta, \beta}$ such that for all $n_0 \in N$, the following inequalities are true.
\begin{eqnarray*}
&& (a) C_{\alpha,\beta} \frac{|\mu|(B(n_0,a))}{m(B(n_0,a))} \le \mathcal{Q}_{i \beta}[|\mu|](\gamma_{n_0,\alpha,\omega,\zeta,\theta}(a)) \:, \: \text{for all} \: a>0\:, \omega \in S^{2p-1} ,\: \zeta \in S^{k-1}, \: \theta \in [0,2\pi) \:. \\
&& (b) \displaystyle\sup_{(n,a) \in \Gamma_{\eta}(n_0)} \left| \mathcal{Q}_{i \beta}[\mu](n,a) \right| \le C'_{\eta,\beta} M_{HL}(\mu)(n_0) \:.
\end{eqnarray*}
\end{lemma}
\begin{proof}
$(b)$ follows from \cite[Lemma 3.3]{Ray} and the trivial fact $\left|\mathcal{Q}_{i \beta}[\mu]\right| \le \mathcal{Q}_{i \beta}[|\mu|]$.
So we will just prove $(a)$. First it will be done for $n_0=\underline{0}$. 

Choose and fix $\alpha >0, \omega \in S^{2p-1}, \zeta \in S^{k-1}, \theta \in [0,2\pi)$, and consider,
\begin{eqnarray}\label{HL_proof_eqn1}
\mathcal{Q}_{i \beta}[|\mu|](\gamma_{\underline{0},\alpha,\omega,\zeta,\theta}(a)) &=& \mathcal{Q}_{i \beta}[|\mu|](\sqrt{\alpha a | \cos \theta |} \omega, \frac{1}{4} \alpha a |\sin \theta| \zeta,a) \nonumber\\
& =& c_{\beta} a^{-Q} \bigintssss_{N} \frac{1}{I_1} \: d|\mu|(X,Z) \nonumber\\
& \ge & c_{\beta} a^{-Q} \bigintssss_{B(\underline{0},a)} \frac{1}{I_1} d|\mu|(X,Z)  \:,
\end{eqnarray}
where
\begin{equation*}
I_1= \left(16+ \frac{8}{a}\|\sqrt{\alpha a | \cos \theta |} \omega - X\|^2 + \frac{1}{a^2} d\left((X,Z)^{-1} \left(\sqrt{\alpha a | \cos \theta |} \omega, \frac{1}{4} \alpha a |\sin \theta| \zeta\right)\right)^2\right)^{\beta + \rho} \:.
\end{equation*}
Now,
\begin{eqnarray}
\label{HL_proof_eqn2}
&& d\left((X,Z)^{-1} \left(\sqrt{\alpha a | \cos \theta |} \omega, \frac{1}{4} \alpha a |\sin \theta| \zeta\right)\right)^2 \nonumber \\
& \le & \left[\tau_d \left\{d\left((X,Z)^{-1}\right) + d\left(\sqrt{\alpha a |\cos \theta| } \omega, \frac{1}{4} \alpha a |\sin \theta| \zeta\right)\right\}\right]^2 \nonumber \\
& = & \left[\tau_d \left\{d(X,Z) + d\left(\sqrt{\alpha a |\cos \theta| } \omega, \frac{1}{4} \alpha a |\sin \theta| \zeta\right)\right\}\right]^2 \:,
\end{eqnarray}
and by (\ref{norm})
\begin{equation}
\label{HL_proof_eqn3}
d\left(\sqrt{\alpha a |\cos \theta| } \omega, \frac{1}{4} \alpha a |\sin \theta| \zeta\right) = \sqrt{2} \alpha a \:.
\end{equation}
Now plugging (\ref{HL_proof_eqn3}) in (\ref{HL_proof_eqn2}) and using $d(X,Z) <a$, we obtain,
\begin{equation}
\label{HL_proof_eqn4}
d\left((X,Z)^{-1} \left(\sqrt{\alpha a | \cos \theta |} \omega, \frac{1}{4} \alpha a |\sin \theta| \zeta\right)\right)^2 \le  \left[\tau_d (a + \sqrt{2} \alpha a) \right]^2 = C^2_{\alpha} \tau^2_d a^2 \:,
\end{equation}
where $C_{\alpha}=(1+\sqrt{2}\alpha)$. Then (\ref{HL_proof_eqn4}) yields,
\begin{equation*}
\|\sqrt{\alpha a |\cos \theta| } \omega - X \|^4 \le d\left((X,Z)^{-1} \left(\sqrt{\alpha a | \cos \theta |} \omega, \frac{1}{4} \alpha a |\sin \theta| \zeta\right)\right)^2 \le C^2_{\alpha} \tau^2_d a^2
\end{equation*}
that is,
\begin{equation} \label{HL_proof_eqn5}
\|\sqrt{\alpha a |\cos \theta| } \omega - X \|^2 \le C_{\alpha} \tau_d a \:.
\end{equation}
Thus plugging (\ref{HL_proof_eqn4}) and (\ref{HL_proof_eqn5})  in (\ref{HL_proof_eqn1}), we get,
\begin{eqnarray*}
&& \mathcal{Q}_{i \beta}[|\mu|](\gamma_{\underline{0},\alpha,\omega,\zeta,\theta}(a)) \\
& \ge & c_{\beta} a^{-Q} \bigintssss_{B(\underline{0},a)} \frac{1}{\left(16 + 8 C_{\alpha} \tau_d + C^2_{\alpha} \tau^2_d \right)^{\rho + \beta}} \\
& = & \frac{c_{\beta} m(B(\underline{0},1))}{\left(16 + 8 C_{\alpha} \tau_d + C^2_{\alpha} \tau^2_d \right)^{\rho + \beta}} \frac{|\mu |(B(\underline{0},a))}{m(B(\underline{0},a))} \:.
\end{eqnarray*}
This completes the proof when $n_0 = \underline{0}$. Now for a general $n_0 \in N$, we consider for $a>0$,
\begin{equation*}
\mathcal{Q}_{i \beta}[|\mu|](\gamma_{n_0,\alpha,\omega,\zeta,\theta}(a)) \:.
\end{equation*}
For simplicity, in the rest of the proof we write an element $s \in S$ as
\begin{equation*}
s=(N(s),a) \: \text{ with } N(s) =(X(s),Z(s)) \:.
\end{equation*}
Then we get,
\begin{equation}\label{HL_proof_eqn6}
\mathcal{Q}_{i \beta}[|\mu|](\gamma_{n_0,\alpha,\omega,\zeta,\theta}(a)) = \mathcal{Q}_{i \beta}[|\mu|](n_0(N(\gamma_{\underline{0},\alpha,\omega,\zeta,\theta}(a))),a) \ge  c_{\beta} a^{-Q} \bigintssss_{B(n_0,a)} \frac{1}{I_2} d|\mu|(X_1,Z_1) \:,
\end{equation}
where
\begin{equation*}
I_2=\left(16+ \frac{8}{a} \|X(n_0(N(\gamma_{\underline{0},\alpha,\omega,\zeta,\theta}(a))))-X_1\|^2 + \frac{1}{a^2} d\left((X_1,Z_1)^{-1} n_0(N(\gamma_{\underline{0},\alpha,\omega,\zeta,\theta}(a)))\right)^2  \right)^{\rho + \beta} \:.
\end{equation*}
Using the fact $(X_1,Z_1) \in B(n_0,a)$ means $n^{-1}_0 (X_1,Z_1)  \in B(\underline{0},a)$ and also using the relations ${(X_1,Z_1)}^{-1}n_0={\left(n^{-1}_0 (X_1,Z_1)\right)}^{-1}$\:,
\begin{equation*}
d\left((X_1,Z_1)^{-1} n_0(N(\gamma_{\underline{0},\alpha,\omega,\zeta,\theta}(a)))\right) = d\left((n^{-1}_0(X_1,Z_1))^{-1} N(\gamma_{\underline{0},\alpha,\omega,\zeta,\theta}(a))\right) \:,
\end{equation*}
from (\ref{HL_proof_eqn4}), we get,
\begin{equation} \label{HL_proof_eqn7}
d\left((X_1,Z_1)^{-1} n_0(N(\gamma_{\underline{0},\alpha,\omega,\zeta,\theta}(a)))\right)^2 \le C^2_{\alpha} \tau^2_d a^2 \:,
\end{equation}
which then in turn implies that
\begin{equation} \label{HL_proof_eqn8}
\|X(n_0(N(\gamma_{\underline{0},\alpha,\omega,\zeta,\theta}(a))))-X_1\|^2 \le C_{\alpha} \tau_d a \:.
\end{equation}
Hence plugging (\ref{HL_proof_eqn7}) and (\ref{HL_proof_eqn8}) in (\ref{HL_proof_eqn6}), and noting that
\begin{equation*}
a^{-Q}= \frac{m(B(\underline{0},1))}{m(B(n_0,a))} \:,
\end{equation*}
we get,
\begin{equation*}
\mathcal{Q}_{i \beta}[|\mu|](\gamma_{n_0,\alpha,\omega,\zeta,\theta}(a)) \ge \frac{c_{\beta} m(B(\underline{0},1))}{\left(16 + 8 C_{\alpha} \tau_d + C^2_{\alpha} \tau^2_d \right)^{\rho + \beta}} \frac{|\mu |(B(n_0,a))}{m(B( n_0,a))} \:.
\end{equation*}
This completes the proof.
\end{proof}
The proof of Lemma 2.5 of \cite{Ray}, yields the following stronger result.
\begin{lemma}
\label{integral_finite}
Let $\beta>0$ and $\mu$ be a radon measure on $N$ such that $|\mu| \ast q^{i\beta}_a(\underline{0})$ is finite for some $a\in A$. Then
\begin{equation*}
\bigintsss_{N} \left(16a^2 + \frac{d(n)^2}{4 \tau^2_d}\right)^{- \beta - \rho} d|\mu|(n) < \infty \:.
\end{equation*}
\end{lemma}
Next we see an asymptotic result.
\begin{lemma}
\label{asympt}
For $\beta >0, \mu \in \M_{\beta}$ and  $ \delta \in (0,1)$, we have,
\begin{equation*}
\displaystyle\lim_{a \to 0} \bigintssss_{{B(\underline{0}, \delta)}^c} q^{i \beta}_a (n^{-1}_1 n) d\mu(n_1)=0 \:,
\end{equation*}
uniformly for $n \in B \left(\underline{0}, \frac{\delta^2}{2 \tau_d}\right)$.
\end{lemma}
\begin{proof}
We note that for $n \in B \left(\underline{0}, \frac{\delta^2}{2 \tau_d}\right)$ and $n_1 \in {B(\underline{0}, \delta)}^c$,
\begin{equation*}
d(n) < \frac{\delta^2}{2 \tau_d} \le  \frac{\delta}{2 \tau_d} d(n_1) \:.
\end{equation*}
Hence,
\begin{equation}
\label{asympt_eqn1}
d(n^{-1}_1n)= d(n^{-1}n_1) \ge \frac{1}{\tau_d} d(n_1) - d(n) \ge \frac{\left(1-\delta/2\right)}{\tau_d} d(n_1)
\end{equation}
Now as $0 < \delta <1$, it follows that,
\begin{equation*}
\frac{1}{4} < \left(1- \frac{\delta}{2}\right)^2 < 1 \:.
\end{equation*}
Now plugging the above in (\ref{asympt_eqn1}), we get,
\begin{equation}
\label{asympt_eqn2}
d(n^{-1}_1n)^2 \ge  \frac{d(n_1)^2}{4 \tau^2_d} \:.
\end{equation}
Now using (\ref{asympt_eqn2}), we get,
\begin{eqnarray*}
&& \left| \bigintssss_{{B(\underline{0}, \delta)}^c} q^{i \beta}_a (n^{-1}_1 n) d\mu(n_1) \right| \\
& \le & \bigintssss_{{B(\underline{0}, \delta)}^c} q^{i \beta}_a (n^{-1}_1 n) d|\mu|(n_1) \\
& = & c_{\beta} a^{2 \beta} \bigintssss_{{B(\underline{0}, \delta)}^c} \frac{1}{\left(16a^2 + 8a \|X-X_1\|^2 +d\left(\left(X_1,Z_1\right)^{-1} \left(X,Z\right)\right)^2\right)^{\rho + \beta}}  d|\mu|(X_1,Z_1)\\
& \le & c_{\beta} a^{2 \beta} \bigintssss_{{B(\underline{0}, \delta)}^c} \frac{1}{\left(16a^2  +d\left(n^{-1}_1n\right)^2 \right)^{\rho + \beta}} d |\mu|(n_1) 
\end{eqnarray*}
\begin{eqnarray*}
& \le & c_{\beta} a^{2 \beta} \bigintssss_{{B(\underline{0}, \delta)}^c} \frac{1}{\left(16a^2  + \frac{d(n_1)^2}{4 \tau^2_d}\right)^{\rho + \beta}} d |\mu|(n_1) \\
& \le & c_{\beta} a^{2 \beta} \bigintssss_{N} \frac{1}{\left(16a^2  + \frac{d(n_1)^2}{4 \tau^2_d}\right)^{\rho + \beta}} d |\mu|(n_1) \:.\\
\end{eqnarray*}
Now as $\mu \in \M_{\beta}$, by Lemma \ref{integral_finite}, the integral on the right hand side above is finite, and thus letting $a \to 0$, we get the result.
\end{proof}
We now need a result about zero sets of real-analytic functions, [Theorem 1.2, \cite{Peralta}]. Although the result was proved for harmonic functions, the proof works for real-analytic functions as well.
\begin{lemma}
\label{peralta}
Let $u$ be a real-analytic function in a domain $D \subset \mathbb R^k$. If $M$ is a real-analytic submanifold of $D$ and the zero set $\mathcal Z$ of $u$ contains a non-empty, relatively open subset of $M$, then $M \subset \mathcal Z \:.$
\end{lemma}
Now as $\mathcal{L}^{\beta}$ is an elliptic operator with no zero order term, we get that $\mathcal{L}^{\beta}$ has a global expression of the form (with $N$ viewed as $\R^{2p+k}$),
\begin{equation*}
\mathcal{L}^{\beta} = \displaystyle\sum_{i,j=0}^{2p+k} c_{ij}(n,a) \frac{\partial^2}{\partial y_i \partial y_j} + \displaystyle\sum_{i=0}^{2p+k} b_i(n,a) \frac{\partial}{\partial y_i} \:\:,
\end{equation*}
on $S$, in some coordinates $\{y_i\}$, with the above coefficients being sufficiently regular and the matrices $\{c_{ij}(n,a)\}_{1 \le i,j \le 2p+k}$ positive definite, symmetric. Then proceeding as in the proof of Theorem 1 in [\cite{Evans},P. 344-346], we get the following maximum principle.
\begin{lemma}
\label{max1}
Let $\Omega$ be an open, connected, bounded subset of $S$. Assume $u \in C^2(\Omega) \cap C(\overline{\Omega})$ (real-valued) and $\mathcal{L}^\beta u \ge 0$ in $\Omega$, then
\begin{equation*}
\displaystyle\max_{\overline{\Omega}} u=\displaystyle\max_{\partial \Omega} u \:.
\end{equation*}
 \end{lemma}

\section{Main results}

\begin{remark} \label{conditions_on_the_boundary_measure}
Now we discuss and compare the conditions on the boundary measure. Our main results in this section will be concerned with measures $\mu \in \mathcal{M}_{\beta}$, such that the concerned convolution satisfy finite radial limit along some ray, that is,
\begin{equation}
\label{radial_limit_finite}
\displaystyle\lim_{a \to 0} \mathcal{Q}_{i \beta}[\mu](\gamma_{n_0,\alpha,\omega,\zeta,\theta}(a)) \:\: \text{is finite } \:,
\end{equation}
for some $n_0 \in N$, $\alpha >0$, $\omega \in S^{2p-1}$, $\zeta \in S^{k-1}$, $\theta \in [0, 2 \pi)$. Now consider the following three additional conditions.
\begin{enumerate}
\item[($\mathscr{H}_1$)] There exists $\varepsilon>0$ such that
\begin{equation*}
\displaystyle\sup_{0<r< \varepsilon} \frac{|\mu|(B(n_0,r))}{m(B(n_0,r))} < \infty \:.
\end{equation*}
\item[($\mathscr{H}_2$)] $\mu$ is positive.
\item[($\mathscr{H}_3$)] In some neighbourhood $V$ of $n_0$ in $S$,
\begin{equation*}
\displaystyle\sup_{(n,a) \in V} \mathcal{Q}_{i \beta}[|\mu|] (n,a) - \left| \mathcal{Q}_{i \beta}[\mu] (n,a) \right| < \infty \:.
\end{equation*}
\end{enumerate}

We note that ($\mathscr{H}_1$) is the hypothesis assumed in Theorem \ref{result_1} and is similar to what Saeki had considered in \cite{Sa}, whereas ($\mathscr{H}_3$) is an obvious analogue of the condition introduced in \cite{Brossard}. Under the assumption (\ref{radial_limit_finite}), we will show that ($\mathscr{H}_1$) is the weakest assumption on the (complex) measure $\mu$. Proof of corollary \ref{result_2}, will show that ($\mathscr{H}_2$) implies ($\mathscr{H}_1$). It is clear from the definitions that ($\mathscr{H}_2$) implies ($\mathscr{H}_3$), as the concerned difference in the definition of ($\mathscr{H}_3$) is identically zero, because in this case $\mu = |\mu|$. It is also not at all hard to see that ($\mathscr{H}_3$) implies ($\mathscr{H}_1$). This is seen as follows, there exists $\delta_1 >0$ such that
\begin{equation*}
\left\{\gamma_{n_0,\alpha,\omega,\zeta,\theta}(a)| 0 < a < \delta_1 \right\} \subset V \:,
\end{equation*}
where $V$ is as given in ($\mathscr{H}_3$). Thus by ($\mathscr{H}_3$),  there exists $M_1>0$ such that
\begin{equation}
\label{remark_eqn1}
\displaystyle\sup_{0 < a < \delta_1} \left\{\mathcal{Q}_{i \beta}[|\mu|] (\gamma_{n_0,\alpha,\omega,\zeta,\theta})(a)) - \left| \mathcal{Q}_{i \beta}[\mu] (\gamma_{n_0,\alpha,\omega,\zeta,\theta})(a)) \right|\right\} \le M_1 \:.
\end{equation}
Then by (\ref{radial_limit_finite}), there exists $\delta_2 >0$ and $M_2 >0$ such that
\begin{equation}
\label{remark_eqn2}
\displaystyle\sup_{0 < a < \delta_2} \left|\mathcal{Q}_{i \beta}[\mu](\gamma_{n_0,\alpha,\omega,\zeta,\theta}(a)) \right|  \le M_2 \:.
\end{equation}
Thus for $\delta=\min\{\delta_1,\delta_2\}$, combining (\ref{remark_eqn1}) and (\ref{remark_eqn2}), we get
\begin{equation}
\label{remark_eqn3}
\displaystyle\sup_{0 < a < \delta} \mathcal{Q}_{i \beta}[|\mu|](\gamma_{n_0,\alpha,\omega,\zeta,\theta}(a)) \le M_1 + M_2 \:.
\end{equation}
Then (\ref{remark_eqn3}) combined with part $(a)$ of Lemma \ref{HL}, completes the proof.
\end{remark}
Now we come to the proof of Theorem \ref{result_1}.
\begin{proof}[Proof of Theorem \ref{result_1}]
We break the proof into two steps. 

\textit{\textbf{Step 1:} Assume $\mu$ to be finite and $n_0=\underline{0}$.} 

We argue by contradiction. Let $\alpha >0$ and suppose there exists a sequence $\{(n_j,a_j)\}_{j=1}^{\infty} \subset \Gamma_{\alpha}(\underline{0})$
such that the sequence $a_j$ converges $0$, (this implies that $(n_j,a_j)$ converges to $(\underline{0},0)$), but $\{u(n_j,a_j)\}_{j=1}^{\infty}$ does not converge to $L$. 

As $\mu$ is finite, we have for all $r \ge t_0$,
\begin{equation}
\label{res1_eqn1}
\frac{|\mu|(B(\underline{0},r))}{m(B(\underline{0},r))} \le \frac{|\mu|(N)}{m(B(\underline{0},t_0))} < \infty \:.
\end{equation}
Then combining (\ref{measure_condition}) and (\ref{res1_eqn1}), we get
\begin{equation*}
M_{HL}(\mu)(\underline{0}) < \infty \:.
\end{equation*}
Then by part $(b)$ of Lemma \ref{HL}, it follows that $u$ is bounded on each admissible region with vertex at $\underline{0}$. Thus in particular, $\{u(n_j,a_j)\}_{j=1}^{\infty}$ is a bounded sequence of complex numbers. So we consider any convergent subsequence with limit say $L'$, and then by showing that $L'=L$, we get the desired contradiction. For notation convenience, let us call this convergent subsequence $\{u(n_j,a_j)\}_{j=1}^{\infty}$ again. For all $j \in \N$, we define
\begin{equation}
\label{res1_eqn3}
u_j(n,a):= u(\delta_{a_j}(n),a_j a) \:, \text{ for all } (n,a) \in S \:.
\end{equation}
Then by Lemma \ref{dilationof_Lbeta}, we get that each $u_j$ is again an $\mathcal{L}^{\beta}$-harmonic function. Now as each admissible region with vertex at $\underline{0}$, is invariant under the natural action of $A$ (see remark \ref{invarianceofadmissible}), we have for all $\eta >0$,
\begin{equation}
\label{res1_eqn4}
\displaystyle\sup_{j \in \N}\sup_{(n,a) \in \Gamma_{\eta}(\underline{0})} \left|u_j(n,a)\right| \le \displaystyle\sup_{(n,a) \in \Gamma_{\eta}(\underline{0})} \left|u(n,a)\right|  < \infty \:.
\end{equation}
As any compact subset of $S$ is contained in an admissible region [see Theorem 4.1,\cite{Ray}], it follows from (\ref{res1_eqn4}), that $\{u_j\}_{j=1}^\infty$ is a locally bounded sequence of $\mathcal{L}^{\beta}$-harmonic functions. Then by Lemma \ref{our_bar}, $\{u_j\}_{j=1}^\infty$ has a subsequence, say, $\{u_{j_k}\}_{k=1}^\infty$, that converges normally to an $\mathcal{L}^{\beta}$-harmonic function, say $v$. It follows from (\ref{res1_eqn4}) that $v$ is also bounded on each admissible region with vertex at $\underline{0}$. 

Now let $\mathscr{S}$ be the sector on $\partial \Gamma_{\kappa}(\underline{0})$, for some $\kappa >0$, on which the function $u$ satisfies the sectorial limit hypothesis. Given any point $(n,a) \in \mathscr{S}$, there exists a ray $\gamma_{\underline{0},\kappa, \omega, \zeta, \theta}$ \:, for some $\omega \in S^{2p-1}, \zeta \in S^{k-1}, \theta \in [0,2 \pi)$, such that
\begin{equation*}
\gamma_{\underline{0},\kappa, \omega, \zeta, \theta}(a)=(n,a) \:.
\end{equation*}
Then by the computation done in remark \ref{invarianceofadmissible} and using the sectorial limit hypothesis, it follows that
\begin{eqnarray}
\label{res1_eqn5}
 v(n,a) =  \displaystyle\lim_{k \to \infty} u_{j_k}(n,a) &= & \displaystyle\lim_{k \to \infty} u_{j_k}(\gamma_{\underline{0},\kappa, \omega, \zeta, \theta}(a)) \nonumber \\
& = & \displaystyle\lim_{k \to \infty} u(a_{j_k} \cdot \gamma_{\underline{0},\kappa, \omega, \zeta, \theta}(a)) \nonumber \\
& = & \displaystyle\lim_{k \to \infty} u(\gamma_{\underline{0},\kappa, \omega, \zeta, \theta}(a_{j_k}a)) \nonumber \\
& = & L \:.
\end{eqnarray}
As $(n,a) \in \mathscr{S}$ was arbitrary, from (\ref{res1_eqn5}), it follows that $v$ identically equals to $L$ on $\mathscr{S}$.
Next set
\begin{equation}
\label{res1_eqn7}
\tilde{v}(n,a) := v(n,a) - L \:, \text{ for } (n,a) \in S \:.
\end{equation}
We write $\tilde{v} = Re(\tilde{v}) + i Im(\tilde{v})$. Note that both $Re(\tilde{v})$ and $Im(\tilde{v})$ are real-valued $\mathcal{L}^\beta$-harmonic functions and hence are real-analytic functions on $S$. Thus by definition (\ref{res1_eqn7}), we get that $\tilde{v}$ vanishes identically on $\mathscr{S}$. So $Re(\tilde{v})$ and $Im(\tilde{v})$ are also identically zero on $\mathscr{S}$. As $S$ is diffeomorphic to Euclidean upper half-space via an analytic diffeomorphism, $\partial \Gamma_{\kappa}(\underline{0})$ is a real-analytic submanifold, with sector being a relatively open subset of $\partial \Gamma_{\kappa}(\underline{0})$, Lemma \ref{peralta} now implies that $Re(\tilde{v})$ identically equals to zero on $\partial \Gamma_{\kappa}(\underline{0})$. The same is also true for $Im(\tilde{v})$. Now as $v$ (being limit of $u_{j_k}$) is bounded on each admissible region with vertex at $\underline{0}$, and hence so is $\tilde{v}$ and thus so are $Re(\tilde{v})$ and $Im(\tilde{v})$. So in particular for $\Gamma_{\kappa}(\underline{0})$, there exists $M>0$ such that
\begin{equation}
\label{res1_eqn9}
\displaystyle \sup_{(n,a) \in \Gamma_{\kappa}(\underline{0})} \left|Re(\tilde{v})\right| \le M \:.
\end{equation}
Now we choose and fix $\varepsilon \in (0,1)$ and $ \delta \in (0,1)$, and define for each $(n,a) \in S$
\begin{equation}
\label{res1_eqn10}
w(n,a):= Re(\tilde{v})(n,a) - M \varepsilon a^{2 \beta} - 2 M \bigintssss_{B(\underline{0}, \delta)} q^{i \beta}_a (n^{-1}_1 n) dm(n_1) \:.
\end{equation}
Here
\begin{equation*}
(n,a) \mapsto \bigintssss_{B(\underline{0}, \delta)} q^{i \beta}_a (n^{-1}_1 n) dm(n_1) \;,
\end{equation*}
is just the convolution of the kernel $q^{i \beta}_a$ with the indicator function of the ball $B(\underline{0}, \delta)$ on $N$, and hence defines a positive $\mathcal{L}^{\beta}$-harmonic function. Clearly,
\begin{equation*}
(n,a) \mapsto a^{2 \beta}  \:,\: \text{for } (n,a) \in S \:,
\end{equation*}
is also a positive $\mathcal{L}^{\beta}$-harmonic function. So in particular, $w$ (defined in (\ref{res1_eqn10})) is a real-valued $\mathcal{L}^{\beta}$-harmonic function. 

It follows from Lemma \ref{asympt} that there exists $a_{\delta} >0$ such that for all $a \le a_{\delta}$,
\begin{equation}
\label{res1_eqn12}
\bigintssss_{B(\underline{0}, \delta)} q^{i \beta}_a (n^{-1}_1 n) d\mu(n_1) \ge \frac{1}{2} \:,
\end{equation}
uniformly for $n \in B\left(\underline{0},\frac{\delta^2}{2 \tau_d}\right)$ \:. Since we are going to work in $\Gamma_{\kappa}(\underline{0})$, we have to make sure that $n \in B\left(\underline{0},\frac{\delta^2}{2 \tau_d}\right)$ (in order to apply Lemma \ref{asympt}). But this is easily done by noting,
\begin{equation}
\label{res1_eqn13}
\Gamma_{\kappa}(\underline{0}) \cap \left\{(n,a) \in S | a \le \frac{\delta^2}{2 \tau_d \kappa}\right\} \subset \left\{(n,a) \in S | n \in B\left(\underline{0}, \frac{\delta^2}{2 \tau_d}\right), a \le \frac{\delta^2}{2 \tau_d \kappa}\right\} \:.
\end{equation}
So we set,
\begin{equation}
\label{res1_eqn14}
a_{\varepsilon,\delta} = \min \left\{a_{\delta}\:,\:\frac{\delta^2}{2 \tau_d \kappa}\:,\: \varepsilon^{1/2\beta}\right\} \:.
\end{equation}
We now define $\Omega_{\varepsilon, \delta}$ to be the solid region enclosed by
\begin{eqnarray*}
&&  \Gamma_{\kappa}(\underline{0}) \cap \left\{a=\left(\frac{1}{\varepsilon}\right)^{1/2\beta}\right\} \:,\\
&& \partial \Gamma_{\kappa}(\underline{0}) \cap \left\{a_{\varepsilon,\delta} \le a \le \left(\frac{1}{\varepsilon}\right)^{1/2\beta}\right\} \:, \\
&& \Gamma_{\kappa}(\underline{0}) \cap \{a=a_{\varepsilon,\delta}\} \:.
\end{eqnarray*}
Also define, $\Omega_{\varepsilon}$ to be the solid region enclosed by
\begin{eqnarray*}
&&  \Gamma_{\kappa}(\underline{0}) \cap \left\{a=\left(\frac{1}{\varepsilon}\right)^{1/2\beta}\right\} \:,\\
&& \partial \Gamma_{\kappa}(\underline{0}) \cap \left\{\varepsilon^{1/2\beta} \le a \le \left(\frac{1}{\varepsilon}\right)^{1/2\beta}\right\} \:, \\
&& \Gamma_{\kappa}(\underline{0}) \cap \left\{a=\varepsilon^{1/2\beta}\right\} \:.
\end{eqnarray*}
Clearly $\Omega_{\varepsilon} \subset \Omega_{\varepsilon, \delta}$. Now for $(n,a) \in \Gamma_{\kappa}(\underline{0}) \cap \left\{a=\left(\frac{1}{\varepsilon}\right)^{1/2\beta}\right\} $ \:, by (\ref{res1_eqn9}), (\ref{res1_eqn10}), we have,
\begin{equation} \label{res1_eqn15}
 w(n,a)  \le  Re(\tilde{v})(n,a) - M \varepsilon a^{2 \beta} \le M - M \varepsilon \left(\left(\frac{1}{\varepsilon}\right)^{1/2\beta}\right)^{2 \beta} = 0 \:.
\end{equation}
Next for $(n,a) \in \partial \Gamma_{\kappa}(\underline{0}) \cap \left\{a_{\varepsilon,\delta} \le a \le \left(\frac{1}{\varepsilon}\right)^{1/2\beta}\right\}$ \:, we have,
\begin{equation}
\label{res1_eqn16}
w(n,a) \le Re(\tilde{v})(n,a) = 0 \:.
\end{equation}
Finally for $(n,a) \in \Gamma_{\kappa}(\underline{0}) \cap \{a=a_{\varepsilon,\delta}\}$, by (\ref{res1_eqn9}), (\ref{res1_eqn12}), (\ref{res1_eqn13}) and definition of $a_{\varepsilon, \delta}$, we get,
\begin{equation} \label{res1_eqn17}
w(n,a) \le  Re(\tilde{v})(n,a)  - 2 M \bigintssss_{B(\underline{0}, \delta)} q^{i \beta}_a (n^{-1}_1 n) dm(n_1) \le M - \left(2M \times \frac{1}{2}\right) =  0 \:.
\end{equation}
Thus combining (\ref{res1_eqn15})-(\ref{res1_eqn17}) and applying Lemma \ref{max1} on $w$, we get that $w(n,a)$ is smaller than zero for all $(n,a) \in \Omega_{\varepsilon,\delta}$. Then by definition of $a_{\varepsilon, \delta}$ (see (\ref{res1_eqn14})), as $\Omega_{\varepsilon}$ is a subset of $\Omega_{\varepsilon, \delta}$, we get that $w(n,a)$ is smaller than zero for all $(n,a) \in \Omega_{\varepsilon}$. Then by (\ref{res1_eqn10}), we get that for all $(n,a) \in \Omega_{\varepsilon}$
\begin{equation}
\label{res1_eqn18}
Re(\tilde{v})(n,a) \le M \varepsilon a^{2 \beta} + 2 M \bigintssss_{B(\underline{0}, \delta)} q^{i \beta}_a (n^{-1}_1 n) dm(n_1) \:.
\end{equation}
In (\ref{res1_eqn18}), $0< \delta <1$ is arbitrary and hence letting $\delta \to 0$, we get that for all $(n,a) \in \Omega_{\varepsilon}$
\begin{equation}
\label{res1_eqn19}
Re(\tilde{v})(n,a) \le M \varepsilon a^{2 \beta}   \:.
\end{equation}
We note that for any fixed $(n_1,a_1) \in \Gamma_{\kappa}(\underline{0})$, there exists $\varepsilon_1 \in (0,1)$ such that $(n_1,a_1) \in \Omega_{\varepsilon} \:,\: \text{for all} \: \varepsilon \in (0,\varepsilon_1) \:.$
Then (\ref{res1_eqn19}) implies that
\begin{equation*}
\label{res1_eqn20}
Re(\tilde{v})(n_1,a_1) \le \varepsilon M  a^{2 \beta}_1 \:,\: \text{for all} \: 0< \varepsilon < \varepsilon_1 \:.
\end{equation*}
It follows by letting $\varepsilon \to 0$, that
\begin{equation*}
Re(\tilde{v})(n_1,a_1) \le 0 \:.
\end{equation*}
Since, $(n_1,a_1) \in \Gamma_{\kappa}(\underline{0})$ was arbitrary, it follows that $Re(\tilde{v})$ is smaller than zero on $\Gamma_{\kappa}(\underline{0})$. Now starting off with $-Re(\tilde{v})$, in place of $Re(\tilde{v})$, and repeating the argument above we get that $-Re(\tilde{v})$  is smaller than zero on $\Gamma_{\kappa}(\underline{0})$. Hence $Re(\tilde{v})$ vanishes identically on $\Gamma_{\kappa}(\underline{0})$. Real-analyticity of $Re(\tilde{v})$ then implies that it vanishes identically on $S$. Similarly, $Im(\tilde{v})$ vanishes identically on $S$. Thus $\tilde{v}$ vanishes identically on $S$.
Then by (\ref{res1_eqn7}), it follows that $v$ identically equals to $L$ on $S$.
Hence $\{u_{j_k}\}_{k=1}^\infty$ converges to the constant function $L$ normally on $S$. Also note that
\begin{equation*}
u(n_j,a_j)=u\left(\delta_{a_j}\left(\delta_{{a_j}^{-1}} (n_j)\right),a_j\right) = u_j \left(\left(\delta_{{a_j}^{-1}} (n_j)\right),1\right) \:.
\end{equation*}
Now since, $\{(n_j,a_j)\}_{j=1}^\infty \subset \Gamma_{\alpha}(\underline{0})$, it follows that
\begin{equation*}
d\left(\delta_{{a_j}^{-1}} (n_j)\right) = \frac{1}{a_j} d(n_j) < \frac{1}{a_j} \alpha a_j =\alpha \:.
\end{equation*}
 Hence,
\begin{equation*}
\left\{\left(\left(\delta_{{a_j}^{-1}} (n_j)\right),1\right)\right\}_{j=1}^\infty \subset \overline{B(\underline{0}, \alpha)} \times \{1\} \:.
\end{equation*}
Thus in particular, we also have,
\begin{equation*}
\left\{\left(\left(\delta_{{a_{j_k}}^{-1}} (n_{j_k})\right),1\right)\right\}_{k=1}^\infty \subset \overline{B(\underline{0}, \alpha)} \times \{1\} \:.
\end{equation*}
So we have,
\begin{equation*}
 L' = \displaystyle\lim_{j \to \infty} u(n_j,a_j) = \displaystyle\lim_{k \to \infty} u(n_{j_k},a_{j_k}) = \displaystyle\lim_{k \to \infty} u_{j_k} \left(\left(\delta_{{a_{j_k}}^{-1}} (n_{j_k})\right),1\right) = L \:.
\end{equation*}
This is a contradiction. This completes the proof for step 1. 

\textit{\textbf{Step 2:}The general case.} 

We can always reduce matters to step 1. A standard argument using the fact that convolution commutes with translation, along with the left-invariance of the quasi-metric implies that it suffices to prove the result for $n_0=\underline{0}$. We then look at the restriction of $\mu$ in a ball on $N$ and argue exactly as in Theorem $4.2$, \cite{Ray}. This completes the proof of Theorem \ref{result_1}.
\end{proof}
Now we apply the last theorem to obtain the following result for certain positive eigenfunctions of the Laplace-Beltrami operator $\mathcal{L}$, on $S$.
\begin{corollary}
\label{result_2}
Suppose that $u$ is a positive eigenfunction of $\mathcal{L}$ on $S$ with eigenvalue $\beta^2 - \rho^2$, where $\beta >0$, and that $n_0 \in N$. If the function
\begin{equation*}
(n,a) \mapsto a^{\beta - \rho} u(n,a) \:, \:\: (n,a) \in S \:,
\end{equation*}
has sectorial limit $L \in [0,\infty)$ at $n_0$, then it has admissible limit $L$ at $n_0$.
\end{corollary}
\begin{proof}
By defining,
\begin{equation*}
F(n,a):= a^{\beta - \rho} u(n,a) \:, \: (n,a) \in S \:,
\end{equation*}
we see from Lemmas \ref{lbetaharmonic}, \ref{positiveeigen} and (\ref{plamdaandqlamda}) that there exists a unique positive $\mu \in \mathcal{M}_{\beta}$ and a unique non-negative constant $C$ such that
\begin{equation}
\label{res2_eqn2}
F(n,a)= Ca^{2 \beta} + \mathcal{Q}_{i \beta}[\mu](n,a) \:, \: (n,a) \in S \:.
\end{equation}
As we are interested in the limit, when $a$ tends to $0$, we may and do assume that $C=0$ in (\ref{res2_eqn2}). Then by the hypothesis regarding sectorial limit, there exists $\alpha >0$, $\omega \in S^{2p-1}$, $\zeta \in S^{k-1}$, $\theta \in [0,2 \pi)$ such that
\begin{equation*}
\displaystyle\lim_{a \to 0} \mathcal{Q}_{i \beta}[\mu] \left(\gamma_{n_0,\alpha,\omega,\zeta,\theta}(a)\right)=L \:.
\end{equation*}
Thus in particular, there exists $t_0 >0$ such that
\begin{equation*}
\displaystyle\sup_{0<a< t_0} \mathcal{Q}_{i \beta}[\mu] \left(\gamma_{n_0,\alpha,\omega,\zeta,\theta}(a)\right) < \infty \:.
\end{equation*}
Then by $(a)$ of Lemma \ref{HL}, we get
\begin{equation*}
\displaystyle\sup_{0<a< t_0} \frac{\mu(B(n_0,a))}{m(B(n_0,a))} < \infty \:.
\end{equation*}
The result now follows from Theorem \ref{result_1}.
\end{proof}
Corollary \ref{result_2} gives a stronger chracterization (compared to Theorem \ref{result_on_admissiblelimit}) of existence of strong derivative of the boundary measure in terms of the boundary behaviour of the corresponding eigenfunction. The next result illustrates it.
\begin{corollary}
\label{result_3}
Suppose that $u$ is a positive eigenfunction of $\mathcal{L}$ in $S$ with eigenvalue $\beta^2 - \rho^2$, where $\beta >0$, with boundary measure $\mu$, and that $n_0 \in N, L \in [0,\infty)$. Then the strong derivative of $\mu$ at $n_0$ equals $L$ if and only if the function $(n,a) \mapsto a^{\beta - \rho} u(n,a)$ has sectorial limit $L$ at $n_0$ .
\end{corollary}
\begin{proof}
The result follows from the equivalence of the following statements.
\begin{itemize}
\item[i)] The function $(n,a) \mapsto a^{\beta - \rho} u(n,a)$ has sectorial limit $L$ at $n_0$ .
\item[ii)] The function $(n,a) \mapsto a^{\beta - \rho} u(n,a)$ has admissible limit $L$ at $n_0$ .
\item[iii)] The boundary measure $\mu$ has strong derivative $L$ at $n_0$ .
\end{itemize}
$ii)$ implies $i)$ is the trivial direction, since any sector on some $\partial \Gamma_{\alpha}(n_0)$ is contained in any admissible region of the form $\Gamma_{\alpha'}(n_0)$, where $\alpha' > \alpha$. $i)$ implies $ii)$ is true due to Corollary \ref{result_2}. Hence, $i)$ is equivalent to $ii)$. Finally, $ii)$ is equivalent to $iii)$ follows from Theorem \ref{result_on_admissiblelimit}.
\end{proof}

\section{The real hyperbolic space}
As $N$ has been assumed to be noncommutative, the class of harmonic $NA$ groups does not include the degenerate case of the real hyperbolic space. We view the real hyperbolic space in the upper half-space model
\begin{equation*}
\mathbb H^l = \{(x,y) | x \in \mathbb R^{l-1}, y \in (0,\infty)\} \:,\: l \ge 2 \;,
\end{equation*}
equipped with the standard hyperbolic metric
\begin{equation*}
ds^2 = \frac{1}{y^2} \left(dx^2 + dy^2 \right) \:,
\end{equation*}
with the boundary being identified with $\mathbb R^{l-1}$. Note that the notion of admissible convergence here reduces to the usual non-tangential convergence and the rays defined as in definition \ref{sectordefn} become usual rays in the setting of euclidean spaces. Also note that in this case the homogeneous dimension is $Q= 2 \rho = l-1$ , and the corresponding differential operators are given by,
\begin{eqnarray*}
&& \mathcal{L} = y^2 \left(\Delta_{\mathbb R^{l-1}} + \frac{\partial^2}{\partial y^2}\right) - (l-2) y \frac{\partial}{\partial y} \:, \\
&& \mathcal{L}^\beta = y^2 \left(\Delta_{\mathbb R^{l-1}} + \frac{\partial^2}{\partial y^2}\right) - (2 \beta -1) y \frac{\partial}{\partial y} \:, \:\text{ for }\: \beta > 0 \:.\\
\end{eqnarray*}
Proceeding as in section $4$, we can obtain the obvious analogues of all the results. Now, when $l=2$, corollary \ref{result_2} can be proved in the following form,
\begin{theorem}
\label{real_hyp_result1}
Suppose that $u$ is a positive eigenfunction of the Laplace-Beltrami operator on $\mathbb H^2$ with eigenvalue $\beta^2 - \frac{1}{4}$ ,\: for $\beta >0$, and boundary measure $\mu$, and let $x_0 \in \mathbb R$. If the function
\begin{equation*}
(x,y) \mapsto y^{\beta - \frac{1}{2}} u(x,y) \:,
\end{equation*}
admits a radial limit $L$ along two rays at $x_0$, then it has non-tangential limit $L$ at $x_0$.
\end{theorem}
Some similar results were obtained in the context of $\R^{2}_+$ regarding the Poisson integral and the Gauss-Weierstrass integral of measures on $\R$, in \cite{Sa}. Now the most natural question is in the setting of Theorem \ref{real_hyp_result1}, what can be said if two different values are assumed along two rays. Of course there does not exist any non-tangential limit, but as it turns out, we can conclude about radial limits along any such ray for $\mathbb H^2$. Analogous result holds for Poisson integral of positive measures on $\R^2_+$ (see \cite{UR}).
\begin{theorem}
\label{real_hyp_result2}
Suppose that $u$ is a positive eigenfunction of the Laplace-Beltrami operator on $\mathbb H^2$ with eigenvalue $\beta^2 - \frac{1}{4}$ ,\: for $\beta >0$, and boundary measure $\mu$, and let $x_0 \in \mathbb R$. Let $\alpha_1, \alpha_2$ be two distinct real numbers and $L_1,L_2 \in [0,\infty)$ such that
\begin{equation*}
\displaystyle\lim_{y \to 0} y^{\beta - \frac{1}{2}} u(x_0 + \alpha_1 y,y) =L_1 \:, \:\: \text{and} \:\: \displaystyle\lim_{y \to 0} y^{\beta - \frac{1}{2}} u(x_0 + \alpha_2 y,y) =L_2 \:,
\end{equation*}
then for all $\alpha \in \mathbb R$,
\begin{equation*}
\displaystyle\lim_{y \to 0} y^{\beta - \frac{1}{2}} u(x_0 + \alpha y,y) =  \frac{L_2 - L_1}{\alpha_2 - \alpha_1} (\alpha - \alpha_1) +L_1 \:.
\end{equation*}
\end{theorem}
\begin{proof}
We set
\begin{equation*}
L(\alpha) = \frac{L_2 - L_1}{\alpha_2 - \alpha_1} (\alpha - \alpha_1) +L_1 \:,\:\: \alpha \in \R.
\end{equation*}
Just as in the proof of Theorem \ref{result_1}, we only work out the case when $x_0=0$ and $\mu$ is finite. As in the proof of corollary \ref{result_2}, we transfer the problem in terms of $\mathcal{L}^{\beta}$-harmonic functions. Again we argue by contradiction. So suppose there exists $\alpha \in \mathbb R$ such that there exists a sequence $\{(\alpha y_j,y_j)\}_{j=1}^\infty$, converging to $(0,0)$, but the sequence $\mathcal{Q}_{i \beta}[\mu](\alpha y_j,y_j)$ does not converge to $L(\alpha)$. Then proceeding as in the proof of Theorem \ref{result_1}, we get hold of an $\mathcal{L}^{\beta}$-harmonic function $v$, obtained as a normal limit of some subsequence of dilates of $\mathcal{Q}_{i \beta}[\mu]$. Now we consider the change of coordinates
\begin{equation} \label{change_of_coord}
(x,y) \mapsto (\eta y,y)\:\: \text{(for } \eta \text{ varying over } \mathbb R) \:.
\end{equation}
In the above transformed coordinates, the differential operator  $\mathcal{L}^\beta$ takes the form
\begin{equation*}
\mathcal{L}^\beta = \frac{\partial^2}{\partial \eta^2} + y^2 \frac{\partial^2}{\partial y^2} - (2 \beta -1) y \frac{\partial}{\partial y} \:.
\end{equation*}
 Then we note that in the transformed coordinates (\ref{change_of_coord}), the function
\begin{equation*}
F(\eta y,y) := \frac{L_2 - L_1}{\alpha_2 - \alpha_1} (\eta - \alpha_1) +L_1 \:,
\end{equation*}
is an $\mathcal{L}^\beta$-harmonic function in $\mathbb H^2$. Thus
\begin{equation*}
\tilde{v}(x,y) := v(x,y) - F(x,y) \:, \: (x,y) \in \mathbb H^2 \:,
\end{equation*}
is an $\mathcal{L}^\beta$-harmonic function in $\mathbb H^2$, which is by construction, non-tangentially bounded and vanishes on the two rays given in the hypothesis. Then an argument exactly analogous to that of the proof of Theorem \ref{result_1} yields a contradiction.
\end{proof}

\bibliographystyle{amsplain}

\begin{thebibliography}{amsplain}
\bibitem{ADY}  Anker, J-P; Damek, E; Yacoub, C. {\em Spherical analysis on harmonic AN groups}. Ann. Scuola Norm. Sup. Pisa Cl. Sci. (4) 23 (1996), no. 4, 643–679 (1997).
\bibitem{ACD} Astengo, F; Camporesi, R; Di Blasio, B. {\em The Helgason Fourier transform on a class of nonsymmetric harmonic spaces}. Bull. Austral. Math. Soc. 55 (1997), no. 3, 405–424.
\bibitem{B} Bar, C. {\em Some properties of solutions to weakly hypoelliptic equations}. Int. J. Differ. Equ. 2013.
\bibitem{BLU} Bonfiglioli, A; Lanconelli, E; Uguzzoni, F. {\em Stratified Lie groups and potential theory for their sub-Laplacians}. Springer Monographs in Mathematics. Springer, Berlin, 2007.
\bibitem{Brossard} Brossard J, Chevalier L.  {\em Probl\'eme de Fatou ponctuel et d\'erivabilit\'e des mesures.} Acta Math., 1990. 164 ,237–263.
\bibitem{D1} Damek, E. {\em A Poisson kernel on Heisenberg type nilpotent groups}. Colloq. Math. 53 (1987), no. 2, 239–247.
\bibitem{DK} Damek, E; Kumar, P. {\em Eigenfunctions of the Laplace-Beltrami operator on harmonic NA groups}. J. Geom. Anal. 26 (2016), no. 3, 1913–1924.
\bibitem{DR} Damek, E; Ricci, F. {\em Harmonic analysis on solvable extensions of H-type groups}. J. Geom. Anal. 2 (1992), no. 3, 213–248.
\bibitem{Peralta} Enciso A, Peralta-Salas D. {\em Some geometric conjectures in harmonic function theory.} Annali di Mathematica Pura ed Applicata. 2011. 192, pages 49–59 (2013).
\bibitem{Evans} Evans L.C. {\em Partial Differential Equations. Graduate studies in mathematics.}1998. 19(4),7.
\bibitem{Fatou} Fatou, P. {\em S\'eries Trigonom\'etriques et S\'eries de Taylor}. Acta Math.,33(1906), 335–400.
\bibitem{FS} Folland, G.B; Stein, E.M. {\em Hardy spaces on homogeneous groups}. Mathematical Notes, 28. Princeton University Press, Princeton, N.J.
\bibitem{Gehring} Gehring, F.W. {\em The Boundary Behavior and Uniqueness of Solutions of the Heat Equation}. Trans Amer. Math. Soc. Vol. 94, No. 3 (Mar., 1960), pp. 337–364.
\bibitem{Kosym} Kor\'anyi, A. {\em Boundary behavior of Poisson integrals on symmetric spaces.} Trans. Amer. Math. Soc. 140 (1969), 393–409.
\bibitem{KP}  Kor\'anyi, A.; Putz, R.B. {\em Local Fatou theorem and area theorem for symmetric spaces of rank one}. Trans. Amer. Math. Soc. 224 (1976), no. 1, 157–168.
\bibitem{L} Loomis, L.H. {\em The converse of the Fatou theorem for positive harmonic functions}. Trans. Amer. Math. Soc. 53, (1943). 239–250.
\bibitem{UR} Ramey, W; Ullrich, D. {\em On the behavior of harmonic functions near a boundary point}. Trans. Amer. Math. Soc. 305 (1988), no. 1, 207–220.
\bibitem{Ray} Ray S.K, Sarkar J. {\em Fatou Theorem and its converse for positive eigenfunctions of the Laplace-Beltrami operator on Harmonic NA groups.} arXiv:2105.04964v2,2021,https://doi.org/10.48550/arXiv.2105.04964.
\bibitem{Sa} Saeki, S. {\em On Fatou-type theorems for non-radial kernels}. Math. Scand. 78 (1996), no. 1, 133–160.




\end{thebibliography}

\end{document}